\documentclass[11pt,a4paper]{article}
\usepackage[american]{babel}      

\usepackage{graphicx} 
\usepackage{amsmath}
\usepackage{mathrsfs}
\usepackage{amssymb}
\usepackage{xcolor}
\usepackage{mathtools}
\usepackage{amsthm}
\usepackage{comment}
\usepackage{float}
\usepackage{tcolorbox}

\theoremstyle{plain}
\newtheorem{theorem}{Theorem}[section]
\theoremstyle{plain}

\theoremstyle{plain}

\theoremstyle{plain}
\newtheorem{lemma}{Lemma}[section]
\theoremstyle{plain}
\newtheorem{remark}{Remark}[section]
\theoremstyle{plain}

\theoremstyle{plain}

\title{A Kantorovich version of Bernstein-type logarithmic operators}
\author{\textit{Laura Angeloni\thanks{Corresponding author}$\ ^1$, Danilo Costarelli$\ ^1$, Chiara Darielli$\ ^2$} \\
\\
$\ ^1$ Department of Mathematics and Computer Science \\
            University of Perugia\\
       1, Via Vanvitelli, 06123 Perugia, Italy    \\  \\
 $\ ^2$ Department of Mathematics and Computer Science\\ 
 University of Firenze\\ 67, Viale Morgagni, 50134 Firenze,  Italy
       \\  \\
  {\small {\tt laura.angeloni@unipg.it}} - {\small {\tt danilo.costarelli@unipg.it}}\\  {\small {\tt 
 chiara.darielli@unifi.it}} 
  }
\date{}

\begin{document}
\maketitle

\begin{abstract}
In this paper, we introduce a Kantorovich version of the Bernstein-type logarithmic operators. The idea comes from the wide literature concerning exponential polynomials that preserve exponential functions: here, the exponential weights are replaced by logarithmic ones and the corresponding operators preserve the logarithmic functions. The pointwise, the uniform and the $L^p$ convergence are first established. Then, a Voronovskaja-type asymptotic formula is derived: from it, a second-order differential operator naturally arises, allowing the characterization of the corresponding saturation class. Finally, quantitative estimates for the order of approximation are provided in the continuous case, in terms of the modulus of continuity, and, in the $L^p$ case, by means of suitable $K$-functionals.
\end{abstract}
\vskip0.2cm

\noindent {\bf Keywords:} positive linear operators; logarithm
preservation; constructive approximation;  saturation by solving differential
problems; modulus of continuity; K-functionals.

\vskip0.3cm

\noindent {\bf AMS Subjclass:} 46A32, 47L05, 47G10, 41A60, 41A25, 41A30   

\section{Introduction}
Bernstein-type operators have a central role in analysis and approximation theory. The classical Bernstein polynomials, defined for bounded functions $f:[0,1]\rightarrow \mathbb{R}$ as
\begin{align*}
B_nf(x)=B_n(f,x) := \sum_{k=0}^{n} f\left(\frac{k}{n}\right)p_{n,k}(x),
\end{align*} 
where 
\begin{align*}
p_{n,k}(x) := \binom{n}{k} x^k (1-x)^{\,n-k}, \qquad x \in [0,1],
\end{align*}
are indeed classical constructive operators that approximate continuous functions on compact intervals (see, e.g., \cite{lorentz1986bernstein,ditzian1987rate,ConApp,DIV,Bustamante}). Over the years, such positive linear operators have inspired several lines of research. Moreover, they were the starting point for many extensions and generalizations aimed at improving approximation properties or enlarging the domain and the function spaces where the approximation processes are performed.

An important extension is given by the Kantorovich version of $B_n$ (see, e.g., \cite{acu2022generalized}), which is defined, for $n \in \mathbb{N}$, as
\begin{align*}
    {K}_n f(x)= {K}_n (f,x):= \sum_{k=0}^n p_{n,k}(x)(n+1) \int\limits_{\frac{k}{n+1}}^{\frac{k+1}{n+1}} f(t)\,dt, \qquad x \in [0,1].
\end{align*}
The operators ${K}_n$ are obtained through an integral formulation, in which the pointwise values in $B_n$ of the function are replaced by integral means.

Besides this, other remarkable generalizations of the Bernstein polynomials were proposed and studied in the literature, such as the King-type operators (\cite{king,gonska2009general,Popa}), the Szász–Mirakjan-type operators (\cite{duman2007szasz, abel2025asymptotic, abel2025commutativity}), the sampling–type operators (\cite{
akhmedov2021mixed,draganov2025weighted,aiello}) and many others (\cite{gonska1991global,Aldaz,Rahman2019,gupta2020approximation,coroianu2022approximation, duman2022approximation, altomare2024convergence,ACD, costarelli2026strong,acu2024bernstein,natale}).

Among them, a remarkable exponential generalization of the classical Bernstein polynomials was introduced by Aral, 
Cárdenas-Morales and Garrancho in~\cite{BerType}. This construction, referred to as Bernstein-type exponential polynomials, is a special case of a more general family of operators previously studied by Morigi and Neamtu in~\cite{Morigi}. In particular, for a fixed real parameter $\mu > 0$, for bounded functions $f:[0,1]\rightarrow \mathbb{R}$ and $n \in \mathbb{N}$, the operators are defined as
\begin{align*} \displaystyle
    \mathscr{G}_nf(x) = \mathscr{G}_n(f,x):= \sum_{k=0}^n f\left( \frac{k}{n} \right) e^{-\mu k / n}e^{\mu x} p_{n,k}(\tilde a_n(x)),
\end{align*} where
\begin{align*}
    \tilde a_n(x)=\frac{e^{\mu x/n}-1}{e^{\mu/n}-1}, \qquad x \in [0,1],
\end{align*}
interpolates $0$ and $1$, is increasing and convex on \([0,1]\).

Following this line of research, several Kantorovich versions of such exponential operators have been proposed (see, e.g., \cite{aral2019approximation, Aral2022, Angeloni2024}). Among them, in particular, the following operators
\begin{align}\label{Kn}
\mathscr{G}_n^K f(x)= \mathscr{G}_n^K(f,x) := \sum_{k=0}^n e^{\mu x}\, p_{n,k}(\tilde a_{n+1}(x))\,(n+1)
\int\limits_{\tfrac{k}{n+1}}^{\tfrac{k+1}{n+1}} f(t)\, e^{-\mu t}\, dt,
\end{align}
with $x \in [0,1]$, $\mu>0$, $n \in \mathbb{N}$, were introduced and studied (\cite{Angeloni2024}). Here, of course, the pointwise values in ${\cal G}_n$ are replaced by suitable weighted integral means.

More recently, motivated by the exponential case, a new class of Bernstein-type operators based on logarithmic  transformations has been introduced in~\cite{Angeloni2025}. The idea is to replace the exponential weights with logarithmic ones, thereby obtaining a structurally similar yet distinct approximation scheme. 
Specifically, we let $\ln_\mu(x):=\ln(1+\mu+x),$ with $\mu >0$, and for a bounded function $f: [0,1] \to \mathbb{R}$ and $n \in \mathbb{N}$, we define the operator as
\begin{align*}
    \mathscr{L}_nf(x) = \mathscr{L}_n(f,x):=\ln_\mu(x)\sum_{k=0}^n f\left( \frac{k}{n} \right) \frac{1}{\ln_\mu\left(\frac{k}{n}\right)}\hspace{1mm}p_{n,k}(a_n(x)),
\end{align*} where
\begin{align*} \displaystyle
a_n(x) = \frac{\ln\left(1+\frac{x}{n(1+\mu)}\right)}{\ln\left(1+\frac{1}{n(1+\mu)}\right)}, \qquad x\in[0,1].
\end{align*} This function is such that $a_n(0)=0$ and $a_n(1)=1$, it is increasing, concave on \([0,1]\) and $a_n(x)\ge x$ for every $x\in [0,1]$.
The purpose of this paper is to introduce a Kantorovich version, in the direction of \eqref{Kn}, of the logarithmic operators $\mathscr{L}_nf$, defined for $n \in \mathbb{N}$ and $\mu > 0$ by
\[
\mathscr{L}_{n}^{K}f(x) := \ln_\mu(x) \sum_{k=0}^n p_{n,k}(a_{n+1}(x)) (n+1) \int\limits_{\frac{k}{n+1}}^{\frac{k+1}{n+1}} \frac{f(t)}{\ln_\mu(t)}\,dt, 
\]
with $x\in [0,1]$, and to study their main approximation properties. 
The main difference between $\mathscr{L}_{n}$ and $\mathscr{L}_{n}^{K}$ is that the latter ones are suitable to reconstruct not necessarily continuous functions. For instance, in spaces where functions are identified by means of equivalence classes, as for example in Lebesgue spaces, the convergence of $\mathscr{L}_{n}$ in general fails, while by means of $\mathscr{L}_{n}^{K}$ we can obtain approximations also in $L^p$. Therefore, they allow the approximation process to be extended to broader classes of functions.

As in the case of the operators $\mathscr{L}_n$, we first establish pointwise and uniform convergence by means of a constructive approach. In addition, we study the convergence of $\mathscr{L}_n^K$ in $L^p$ spaces, $1\le p < +\infty$, which represents one of the main advancement with respect to \cite{Angeloni2025}.
Furthermore, we derive a Voronovskaja-type formula from which a second-order differential operator naturally arises. The set of the solutions of the associated homogeneous differential equation characterizes the saturation (Favard) class of the operators.

We also obtain quantitative estimates for the error of approximation. In the continuous case, these estimates are expressed in terms of the classical modulus of continuity, while, in the $L^p$ case, they are obtained by means of a Peetre $K$-functional for functions in $C^1([0,1])$, together with a slightly modified version, denoted by $\widetilde{K}$. More precisely, the definition of $\widetilde{K}$ involves the norm of functions belonging to the Sobolev space $W^{1,p}([0,1])$ in place of their seminorm, as happens in the usual case. In order to achieve the estimate in the $L^p$ norm, a crucial role is played by the so-called Hardy–Littlewood maximal inequality.

The above results in the $L^p$-setting, which are typical when working with Kantorovich-type operators,  represent a novelty with respect to \cite{Angeloni2025}. Moreover, the introduction of the K-functional $\widetilde{K}$ represents a key point in the study of the order of approximation of King-type operators. Finally, it can be useful to point out the role of the parameter $\mu>0$ in the definition of $\ln_\mu$ and consequently of $\mathscr{L}_n^K$. First, it guarantees the well-posedness of the definition of the operators; second, it is also connected to some potential applicative motivations.  Indeed, in \cite{Angeloni2025} (see Section 7) we proposed a toy model to reconstruct signals affected by multiplicative-type noise exploiting the linearizing property of the logarithmic function. In this type of application, the parameter $\mu$ plays an important role since it can be viewed as a Gaussian random variable connected to the noise that affects the signal.

\section{Kantorovich-type logarithmic operators}
Motivated by the recent definition of Bernstein-type logarithmic operators, we introduce their Kantorovich counterpart. 

For $n \in \mathbb{N}$, $\mu > 0$ and $f:[0,1]\to \mathbb{R}$ bounded, we define
\begin{align*}
\mathscr{L}_{n}^{K}f(x)=\mathscr{L}_{n}^{K}(f,x) := \ln_\mu(x) \sum_{k=0}^n p_{n,k}(a_{n+1}(x)) (n+1) \int\limits_{\tfrac{k}{n+1}}^{\tfrac{k+1}{n+1}} f_\mu(t)\,dt,
\end{align*}
where
\begin{align}\label{ln}
    \ln_\mu(x):=\ln(1+\mu+x),
\end{align}
and $f_\mu(x):=\frac{f(x)}{\ln_\mu(x)}$, with $x\in[0,1]$. The presence of the parameter $\mu>0$ ensures that $\ln_\mu(x)$ is strictly positive on $[0,1]$, making the operators well-defined. 

The function $a_{n+1}(x)$, that is defined as
\begin{align} \label{an+1}
a_{n+1}(x)=\frac{\ln\left(1+\frac{x}{(n+1)(1+\mu)}\right)}{\ln\left(1+\frac{1}{(n+1)(1+\mu)}\right)},
\end{align} is such that \( a_{n+1}(0) = 0 \) and \( a_{n+1}(1) = 1 \). Moreover, it is increasing and concave on \([0,1]\), with $a_{n+1}(x)\ge x$ for every $x\in [0,1]$. The convergence properties of the sequence $(a_{n+1}(x))_n$ are given in the following lemmas established in \cite{Angeloni2025} (see Lemma 3.1 and equation (17), respectively).

\begin{lemma}\label{lemma_an+1}
Let \(a_{n+1}(x)\) be defined as in \eqref{an+1}. Then, the sequence \( (a_{n+1}(x))_n \) converges uniformly to \( e_1(x) \equiv x \) on \([0,1]\).
\end{lemma}

\begin{lemma}\label{lemma_n(an+1-x)}
Let $a_{n+1}(x)$ be defined as in \eqref{an+1}. Then, there holds
\begin{align*}
\lim_{n \to +\infty} n \,(a_{n+1}(x) - x) = \frac{x - x^2}{2(1+\mu)}, \ \ x\in [0,1].
\end{align*}
\end{lemma}

Concerning the above operators, we now observe that \( (\mathscr{L}_n^K)_n \) are positive linear operators that preserve the logarithmic function \( \ln_\mu(x) \), defined in \eqref{ln}. Indeed,
\begin{align*}
\mathscr{L}_n^{K}\ln_\mu(x) 
&= \ln_\mu(x) \sum_{k=0}^n p_{n,k}(a_{n+1}(x)) (n+1) \int\limits_{\frac{k}{n+1}}^{\frac{k+1}{n+1}} dt
\\&= \ln_\mu(x) \sum_{k=0}^n p_{n,k}(a_{n+1}(x)) = \ln_\mu(x),
\end{align*}
taking into account that $\sum\limits_{k=0}^n p_{n,k}(a_{n+1}(x)) = 1$ for every $x \in [0,1]$. Hence, the operators $\mathscr{L}_n^{K}f(x)$ fix the function $\ln_\mu(x)$.

\section{Convergence results}
In this section, we establish the pointwise and uniform convergence of the sequence of operators \((\mathscr{L}_n^K)_n\) through a constructive approach. Additionally, we study the convergence of the operators in the $L^p$ spaces. We begin with the following.

\begin{theorem}\label{ConvCostr}If $f:[0,1]\rightarrow \mathbb{R}$ is bounded, then 
\begin{align*} \displaystyle 
\lim_{n \rightarrow +\infty} \mathscr{L}_n^Kf(x) =f(x),
\end{align*}
at every point $x\in[0,1]$ where $f$ is continuous. Moreover, for every $f\in C([0,1])$, there holds
\begin{align*}
    \lim_{n\to +\infty} \mathscr{L}_{n}^{K}f(x) = f(x), \quad \text{uniformly \,\,\,for }\,\, x \in [0,1].
\end{align*}
\end{theorem}
\begin{proof}
We focus on proving the uniform convergence, as the pointwise convergence can be obtained using analogous reasonings. \\[0.2em]
Since $f \in C([0,1])$, it is uniformly continuous. 
Therefore, for every $\varepsilon>0$ there exists $\delta(\varepsilon)>0$ such that
\begin{align*}
|f(x)-f(y)| \le \varepsilon,
\end{align*}
for every $x,y \in [0,1]$ with $|x-y| \le \delta$. Let now $\varepsilon>0$ be fixed.
First, we write
\begin{align*}
|\mathscr{L}_n^K f(x) - f(x)| &= \bigl|\mathscr{L}_n^K f(x) - f(x)\,\mathscr{L}_n^K e_0(x)
+ f(x)\,\mathscr{L}_n^K e_0(x) - f(x)\bigl|\\[1em]
&\leq \bigl|\mathscr{L}_n^K f(x) - f(x)\,\mathscr{L}_n^K e_0(x)\bigl|
+ \bigl|f(x)\,\mathscr{L}_n^K e_0(x) - f(x)\bigl|\\[1em]
&=:I_1+I_2.
\end{align*}
Now, by the definition of \(\mathscr{L}_n^K\), we have
\begin{align*}
I_1
&= \left|\,
\ln_\mu(x)\sum_{k=0}^n p_{n,k}(a_{n+1}(x))\,(n+1)
\int\limits_{\frac{k}{n+1}}^{\frac{k+1}{n+1}} f_\mu(t) dt\right.\\[0.5em]
&\left.\hspace{6mm}-f(x)\ln_\mu(x)\sum_{k=0}^n p_{n,k}(a_{n+1}(x))\,(n+1)
\int\limits_{\frac{k}{n+1}}^{\frac{k+1}{n+1}} \frac{1}{\ln_\mu(t)} dt
\right|\\[0.5em]
&\leq \ln_\mu(x)\sum_{k=0}^n p_{n,k}(a_{n+1}(x))\,(n+1)
\int\limits_{\frac{k}{n+1}}^{\frac{k+1}{n+1}} \frac{|f(t)-f(x)|}{\ln_\mu(t)}\,dt.
\end{align*}

Since \(\ln_\mu(t) \ge \ln(1+\mu) > 0\) for every \(t\ge 0\), we obtain
\begin{align*}
I_1 &\le \frac{\ln_\mu(x)}{\ln(1+\mu)}\sum_{k=0}^n p_{n,k}(a_{n+1})\,(n+1)\int\limits_{\frac{k}{n+1}}^{\frac{k+1}{n+1}} |f(t)-f(x)|\, dt\\[0.5em]
&\le \frac{\ln_\mu(x)}{\ln(1+\mu)}\Bigg\{\sum_{\left|x-\frac{k}{n+1}\right|\le\frac{\delta}{2}}+\sum_{\left|x-\frac{k}{n+1}\right|>\frac{\delta}{2}}\Bigg\}\,\, p_{n,k}(a_{n+1}(x))\,(n+1) \\[0.5em]&\hspace{7mm}\cdot\int\limits_{\frac{k}{n+1}}^{\frac{k+1}{n+1}} |f(t)-f(x)|\, dt =: I_{1,1} + I_{1,2},
\end{align*}
where $\delta >0$ is the parameter of the uniform continuity of $f$ correspondingly to $\varepsilon$.
For the sum $I_{1,1}$, by the uniform continuity of $f$, if 
$\left|x - \frac{k}{n+1}\right| \le \frac{\delta}{2}$, then for every 
$t \in \left[\frac{k}{n+1},\frac{k+1}{n+1}\right]$ we have
\[
|x - t| \le \left|x - \frac{k}{n+1}\right| 
      + \left|\frac{k}{n+1} - t\right|
      \le \frac{\delta}{2} + \frac{1}{n+1}\le \delta,
\]
provided that $\frac{1}{n+1} \le \frac{\delta}{2}$. 
Therefore, for $n$ sufficiently large,
\begin{align*}
I_{1,1} &\le \frac{\ln_\mu(x)}{\ln(1+\mu)} \sum_{\left|x - \frac{k}{n+1}\right| \le \frac{\delta}{2}} p_{n,k}(a_{n+1}(x))\,(n+1)
\int\limits_{\frac{k}{n+1}}^{\frac{k+1}{n+1}} \varepsilon \, dt \\[0.5em]
&= \frac{\ln_\mu(x)}{\ln(1+\mu)} \,\varepsilon \sum_{\left|x - \frac{k}{n+1}\right| \le \frac{\delta}{2}} p_{n,k}(a_{n+1}(x)) \le\frac{\ln(2+\mu)}{\ln(1+\mu)} \,\varepsilon.
\end{align*}
Moreover, for the sum $I_{1,2}$, we note that
\begin{align*}
I_{1,2}&\le\frac{2\Vert f\Vert_\infty\ln(2+\mu)}{\ln(1+\mu)}\sum_{\left|x-\frac{k}{n+1}\right|>\frac{\delta}{2}}p_{n,k}(a_{n+1}(x))\\
&= \frac{2\Vert f\Vert_\infty\ln(2+\mu)}{\ln(1+\mu)}\sum_{\left|x-\frac{k}{n+1}\right|>\frac{\delta}{2}}p_{n,k}(a_{n+1}(x))\frac{\left|x-\frac{k}{n+1}\right|}{\left|x-\frac{k}{n+1}\right|}\\[0.5em]
&\le\frac{{4}\Vert f\Vert_{\infty}\ln(2+\mu)}{\delta\ln(1+\mu)}\sum_{\left|x-\frac{k}{n+1}\right|>\frac{\delta}{2}}\left|x-\frac{k}{n+1}\right|\,\,p_{n,k}(a_{n+1}(x)).
\end{align*}
Considering the last term in the above equation we get:
\begin{align*} 
    \sum_{\left|x-\frac{k}{n+1}\right|>\frac{\delta}{2}} \Big|x-\frac{k}{n+1}\Big|\,\,p_{n,k}(&a_{n+1}(x))\le\sum_{\left|x-\frac{k}{n+1}\right|>\frac{\delta}{2}} |x-a_{n+1}(x)|\,\,p_{n,k}(a_{n+1}(x))\\[0.5em]
    &\hspace{7mm}+\sum_{\left|x-\frac{k}{n+1}\right|>\frac{\delta}{2}} \Big|a_{n+1}(x)-\frac{k}{n+1}\Big|\,\,p_{n,k}(a_{n+1}(x)).
\end{align*}
By Lemma \ref{lemma_an+1}, we have that $a_{n+1}(x)\to x$ uniformly, as $n\to +\infty,$ hence there exists $N_1\in \mathbb{N}$ such that, for $n\ge N_1$,
\[
|x-a_{n+1}(x)|<\varepsilon\delta,
\]
for all $x\in [0,1]$. Moreover, by Lemma~5.1 of \cite{Angeloni2024}, which states that
\begin{align} \label{DisLemma5.1}
\sum_{k=0}^n \left|y-\frac{k}{n+1}\right|p_{n,k}(y)<\frac{\sqrt{2}}{\sqrt{n+1}},
\end{align} 
for every $y\in [0,1]$ and $n>1$, it follows that
\begin{align*} \displaystyle
    \sum_{\left|x-\frac{k}{n+1}\right|>\frac{\delta}{2}} \left|a_{n+1}(x)-\frac{k}{n+1}\right|p_{n,k}(a_{n+1}(x))<\frac{\sqrt{2}}{\sqrt{n+1}}.
\end{align*}
Now, taking $N_2$ large enough so that for $n\ge N_2$ we have that $\frac{\sqrt{2}}{\sqrt{n+1}}<\varepsilon\delta$, and setting $N=\max\{N_1,N_2\}$, for $n\ge N$ we obtain
\[
\sum_{|x-\frac{k}{n+1}|>\frac{\delta}{2}} \Big|x-\frac{k}{n+1}\Big|\,p_{n,k}(a_{n+1}(x)) \le 2\varepsilon\delta.
\]
Therefore, 
\begin{align*}
I_1 \le I_{1,1}+I_{1,2}\le \varepsilon\ \frac{\ln(2+\mu)}{\ln(1+\mu)}\ (1+{8}\Vert f\Vert_\infty),
\end{align*}
for $n$ sufficiently large. On the other side,
\begin{align*}
I_2&= |f(x)\mathscr{L}^K_n(e_0,x)-f(x)|\le \Vert f\Vert_\infty|\mathscr{L}^K_n(e_0,x)-e_0(x)|\\[0.5em]
&= \Vert f\Vert_\infty\left|\ln_\mu(x)\sum_{k=0}^n p_{n,k}(a_{n+1}(x))\,(n+1)\int\limits_{\frac{k}{n+1}}^{\frac{k+1}{n+1}} \frac{1}{\ln_\mu(t)}\,dt - 1\right|\\[0.5em]
&\le \Vert f\Vert_\infty\ln_\mu(x)\sum_{k=0}^n p_{n,k}(a_{n+1}(x))\,(n+1)\int\limits_{\frac{k}{n+1}}^{\frac{k+1}{n+1}}
\left|\frac{1}{\ln_\mu(t)} - \frac{1}{\ln_\mu(x)}\right|\,dt,
\end{align*}
since obviously \(\sum\limits_{k=0}^n p_{n,k}(a_{n+1})(n+1)\int\limits_{k/(n+1)}^{(k+1)/(n+1)} \,dt= 1\). Now
\begin{align*}
I_2 &\le \|f\|_\infty \ln_\mu(x)\sum_{k=0}^n p_{n,k}(a_{n+1}(x))\,(n+1)\int\limits_{\frac{k}{n+1}}^{\frac{k+1}{n+1}} \frac{|\ln_\mu(t)-\ln_\mu(x)|}{\ln_\mu(x)\ln_\mu(t)}\,dt
\\[0.5em]
&\le \frac{\|f\|_\infty}{\ln(1+\mu)}\sum_{k=0}^n p_{n,k}(a_{n+1}(x))\,(n+1)\int\limits_{\frac{k}{n+1}}^{\frac{k+1}{n+1}} |\ln_\mu(t)-\ln_\mu(x)|\,dt\\[0.5em]
&\le \frac{\|f\|_\infty}{\ln(1+\mu)}\Bigg\{\sum_{\left|x-\frac{k}{n+1}\right|\le\, \frac{\gamma}{{2}}}+\sum_{\left|x-\frac{k}{n+1}\right|>\,  \frac{\gamma}{{2}}}\Bigg\}\,\, p_{n,k}(a_{n+1}(x))\,(n+1)\end{align*}\begin{align*}
&\hspace{-20mm}\cdot\int\limits_{\frac{k}{n+1}}^{\frac{k+1}{n+1}} |\ln_\mu(t)-\ln_\mu(x)|\,dt=:I_{2,1}+I_{2,2},
\end{align*} where $\gamma$ is the parameter for the uniform continuity of $\ln_\mu$ in $[0,1]$, correspondingly to $\varepsilon$. 
In particular, proceeding as in the first part of the proof, we have
\begin{align*}
I_{2,1}
&\le \frac{\|f\|_\infty}{\ln(1+\mu)}\,\varepsilon
\sum_{\left|x-\frac{k}{n+1}\right|\le \frac{\gamma}{{2}}} p_{n,k}(a_{n+1}(x))
\le \frac{\|f\|_\infty}{\ln(1+\mu)}\,\varepsilon,
\end{align*}
and, for sufficiently large $n\in \mathbb{N}$,
\begin{align*}
I_{2,2}\le \frac{2\|f\|_\infty\ln(2+\mu)}{\ln(1+\mu)}
\sum_{\left|x-\frac{k}{n+1}\right|>\,\frac{\gamma}{{2}}} p_{n,k}(a_{n+1}(x)) \le \frac{{8}\,\|f\|_\infty\ln(2+\mu)}{\ln(1+\mu)}\,\varepsilon.
\end{align*}
Therefore, for every \(x\in[0,1]\) and $n$ sufficiently large, we finally get
\[
I_2 = I_{2,1}+I_{2,2} \le 
\frac{\varepsilon\|f\|_\infty}{\ln(1+\mu)}\Big(1 + 8\ln(2+\mu)\Big).
\] In conclusion, for every \(x\in[0,1]\) and $n$ sufficiently large, we obtain
\[
\bigl|\mathscr{L}_n^{K}f(x)-f(x)\bigr|
\le \varepsilon\frac{\ln(2+\mu)}{\ln(1+\mu)}\Bigg\{1+\|f\|_\infty\left(16+\frac{1}{\ln(2+\mu)}\right)\Bigg\}, 
\] from which the thesis immediately follows.\\
\end{proof}
Now, we study the convergence of $\mathscr{L}^K_n$ in $L^p$. In particular, we consider the space $L^p_\mu([0,1])$, that is a weighted $L^p-$space for $1\le p <+\infty$, i.e., the space of the measurable functions $f:[0,1]\to\mathbb{R}$ such that 
\begin{align*}
    \Vert f\Vert_{p,\mu}:=\left\{\int\limits_0^1 |f_\mu(x)|^pdx\right\}^{\frac{1}{p}}<+\infty,
\end{align*} with $f_\mu(x):=\frac{f(x)}{\ln_\mu(x)}$ as usual. 

\begin{theorem}\label{th_Kmu}
If $f\in L^p_\mu([0,1])$, then
\begin{align} \label{Kmu}
    \|\mathscr{L}_n^K f\|_{p,\mu}^p \le K_\mu \,\|f\|_{p,\mu}^p,
\end{align}
with $K_\mu:=\left(1+\frac{1}{1+\mu}\right)$, and
\begin{align*}
    \Vert\mathscr{L}^K_\mu f-f \Vert_{p,\mu}\to 0, \quad \text{as }n \to+\infty.
\end{align*}
\end{theorem}

\begin{proof}
For $f\in L^p_\mu([0,1])$, we have, by the convexity of $(\cdot)^p$,
\begin{align*}
    \Vert \mathscr{L}_n^Kf\Vert_{p,\mu}^p&=\int\limits_0^1 \left|\frac{\mathscr{L}_n^K f(x)}{\ln_\mu(x)}\right|^pdx\\[0.5em]
    &=\int\limits_0^1 \Bigg|\sum_{k=0}^n p_{n,k}(a_{n+1}(x))(n+1)\int\limits_{\frac{k}{n+1}}^{\frac{k+1}{n+1}}f_\mu(t)\,dt\Bigg|^p dx\\[0.5em]
    &\le \int\limits_0^1 \sum_{k=0}^n p_{n,k}(a_{n+1}(x))\Bigg|(n+1)\int\limits_{\frac{k}{n+1}}^{\frac{k+1}{n+1}}f_\mu(t)\,dt\Bigg|^p dx.
\end{align*}
Now, we can apply the Jensen's inequality in the interval $\left[\frac{k}{n+1},\frac{k+1}{n+1}\right]$, so
\begin{align*}
\Vert \mathscr{L}_n^Kf\Vert_{p,\mu}^p&\le \sum_{k=0}^n \int\limits_0^1 p_{n,k}(a_{n+1}(x))dx\,\, (n+1) \int\limits_{\frac{k}{n+1}}^{\frac{k+1}{n+1}}|f_\mu(t)|^p\,dt\\[0.5em]
&=\sum_{k=0}^n (n+1)\int\limits_{\frac{k}{n+1}}^{\frac{k+1}{n+1}}|f_\mu(t)|^p\,dt\int\limits_0^1 p_{n,k}(a_{n+1}(x))\,dx.
\end{align*}
Now, we estimate the integral
\begin{align*}
    I_{n,k}:=(n+1)\int\limits_0^1 p_{n,k}(a_{n+1}(x))dx,
\end{align*}
with the following change of variable $$t = a_{n+1}(x)=\frac{\ln\left(1+\frac{x}{(n+1)(1+\mu)}\right)}{\ln\left(1+\frac{1}{(n+1)(1+\mu)}\right)},$$ so that
\begin{align*}
    \ln\left(1+\frac{x}{(n+1)(1+\mu)}\right)&=t\cdot\,\ln\left(1+\frac{1}{(n+1)(1+\mu)}\right)\\[0.5em]
    1+\frac{x}{(n+1)(1+\mu)}&=\left(1+\frac{1}{(n+1)(1+\mu)}\right)^t\\[0.5em]
    \hspace{-5mm}x&=(n+1)(1+\mu)\left\{\left(1+\frac{1}{(n+1)(1+\mu)}\right)^t-1\right\},
\end{align*}
and $dx= (n+1)(1+\mu)\cdot\ln\left(1+\frac{1}{(n+1)(1+\mu)}\right)\cdot\left(1+\frac{1}{(n+1)(1+\mu)}\right)^t dt$. In this way, the integral becomes
\begin{align*}
I_{n,k}&=(n+1)^2(1+\mu)\ln\left(1+\frac{1}{(n+1)(1+\mu)}\right)\\[0.5em]&\hspace{5mm}\cdot\int\limits_0^1 p_{n,k}(t)\left(1+\frac{1}{(n+1)(1+\mu)}\right)^t\,dt\\[0.2em]
&\le (n+1)^2(1+\mu)\left(1+\frac{1}{(n+1)(1+\mu)}\right)\ln\left(1+\frac{1}{(n+1)(1+\mu)}\right)\\[0.2em]
&\hspace{5mm}\cdot\int\limits_0^1 p_{n,k}(t)dt\\[-0.6em]
&=(n+1)\Big((n+1)(\mu+1)+1\Big)\ln\left(1+\frac{1}{(n+1)(1+\mu)}\right)\int\limits_0^1 p_{n,k}(t)dt,
\end{align*}
where, for every positive integer $k$ (see \cite{ConApp}):
\begin{align*}
\int\limits_{0}^{1}p_{n,k}(t) dt &= \binom{n}{k} \int\limits_{0}^{1}t^k (1-t)^{n-k}dt = \frac{1}{n+1}.
\end{align*} As a consequence,
\begin{align*}
I_{n,k}&\le\Big((n+1)(\mu+1)+1\Big)\ln\left(1+\frac{1}{(n+1)(1+\mu)}\right)\\[0.5em]
&\le \Big((n+1)(\mu+1)+1\Big)\left(\frac{1}{(n+1)(1+\mu)}\right)\\[0.5em]
&=1+\frac{1}{(n+1)(1+\mu)}, 
\end{align*} 
and so we conclude that
\begin{align*}
\Vert \mathscr{L}_n^K\Vert_{p,\mu}^p &\le \left(1+\frac{1}{(n+1)(1+\mu)}\right)\sum_{k=0}^n \int\limits_{\frac{k}{n+1}}^{\frac{k+1}{n+1}}|f_\mu(t)|^p\,dt\\[0.5em]
&=\left(1+\frac{1}{(n+1)(1+\mu)}\right) \int\limits_{0}^{1}|f_\mu(t)|^p\,dt\\[0.5em]
&=\left(1+\frac{1}{(n+1)(1+\mu)}\right)\Vert f\Vert_{p,\mu}^p,
\end{align*}
for all $n\ge 1$, whence
\begin{align*}
\|\mathscr{L}_n^K f\|_{p,\mu}^p \le \Bigl(1+\frac{1}{1+\mu}\Bigr)\,\|f\|_{p,\mu}^p=:K_\mu\,\|f\|_{p,\mu}^p.
\end{align*} 
This concludes the first part of the proof. Let now $\varepsilon > 0$ be fixed. By the density of $C([0,1])$ in $L^p([0,1])$, there exists $h \in C([0,1])$ such that
\begin{align*}
\|f - h\|_p < \ln(1+\mu)\frac{\varepsilon}{2(K_\mu {^{1/p}} + 1)}.
\end{align*}
As a consequence,
\begin{align*}
    \Vert f -h \Vert_{p,\mu} &= \left\{\int\limits_0^1\left(\frac{|f(t)-h(t)|}{\ln_\mu(t)}\right)^pdt\right\}^{\frac{1}{p}\notag}\\[0.3em]
    &\le\frac{1}{\ln(1+\mu)} \left\{\int\limits_0^1|f(t)-h(t)|^pdt\right\}^{\frac{1}{p}}\\[0.5em]
    &=\frac{\|f-h\|_p}{\ln(1+\mu)}<\frac{\varepsilon}{2(K_\mu   ^{1/p}+1)}.
\end{align*}
Now, 
\begin{align*}
\|\mathscr{L}_n^K f - f\|_{p,\mu} 
&= \|\mathscr{L}_n^K f - \mathscr{L}_n^K h + \mathscr{L}_n^K h - h + h - f\|_{p,\mu} \\[0.5em]
&\le \|\mathscr{L}_n^K(f-h)\|_{p,\mu}+ \|\mathscr{L}_n^K h - h\|_{p,\mu}+ \|h - f\|_{p,\mu}\\[0.5em]
&=:J_1+J_2+J_3.
\end{align*}
By (\ref{Kmu}), we immediately have:
\begin{align*}
J_1 = \|\mathscr{L}_n^K(f-h)\|_{p,\mu} \le K_\mu^{1/p}\, \|f-h\|_{p,\mu}.
\end{align*}
For the second term $J_2$, since $h \in C([0,1])$, we can use Theorem \ref{ConvCostr} to obtain 
\begin{align*}
J_2=\|\mathscr{L}_n^K h - h\|_{p,\mu} \le \left\{ \int\limits_0^1 \left| \mathscr{L}_n^kh(x)-h(x)\right|^p\right\}^\frac{1}{p}\le \|\mathscr{L}_n^Kh-h\|_\infty< \frac{\varepsilon}{2}, 
\end{align*} for $n$ sufficiently large. Hence, 
\begin{align*}
\|\mathscr{L}_n^K f - f\|_{p,\mu} &\le (K_\mu^{1/p}+1)\|f-h\|_{p,\mu}+\frac{\varepsilon}{2} \\[0.5em]&< (K_\mu^{1/p}+1)\frac{\varepsilon}{2(K_\mu^{1/p}+1)}+\frac{\varepsilon}{2}=\varepsilon,
\end{align*} for $n$ sufficiently large. This completes the proof.\\[0.2em]
\end{proof} 
We point out that
\begin{align}\label{stima_normap}
    \frac{1}{\ln(2+\mu)}\|f\|_p \le \|f\|_{p,\mu}\le \frac{1}{\ln(1+\mu)}\|f\|_p,
\end{align} and therefore Theorem~\ref{th_Kmu} can equivalently be formulated in terms of the usual $L^p$-norm $\|\cdot\|_p$, multiplied by an explicit constant depending only on $\mu$, instead of the weighted norm $\|\cdot\|_{p,\mu}$, for every $1\le p<+\infty$.

\section{Voronovskaja-type formula}
In this section, we establish a Voronovskaja-type asymptotic formula for the operators $\mathscr{L}_n^K$. 

\begin{theorem}\label{Vor}
If $f \in C^2([0,1])$, then
\begin{align*} \displaystyle 
\lim_{n \to +\infty} n\left[ \mathscr{L}^K_n(f,x) - f(x) \right] 
=& \ln_\mu(x)\left[f_\mu'(x)\left(\frac{1}{2}+\frac{x-x^2}{2(1+\mu)}-x\right)\right.\\[0.5em]&\left.+f_\mu''(x)\left(\frac{x-x^2}{2}\right)\right],
\end{align*}  for $x\in[0,1]$.
\end{theorem}
\begin{proof}
First, we note that, if $f \in C^2([0,1])$, then $f_\mu \in C^2([0,1])$ as well. We now fix a point \( x \in [0,1] \) and apply the second-order Taylor formula with Peano remainder to \( f_\mu \). In particular, we obtain
\[
f_\mu\left( t \right) = f_\mu(x) + f_\mu'(x) \left( t - x \right) + \left( t - x \right)^2 \left[\frac{1}{2} f_\mu''(x) + \psi\left( t - x \right)\right],
\]
where \( \psi(y) \) is a bounded function for every \( y\) that tends to zero as \( y \to 0^+ \). Replacing this expansion into the operator, we get
\begin{align*}
\mathscr{L}^K_n(f,x) &=\,\ln_{\mu}(x) \sum_{k=0}^n p_{n,k}(a_{n+1}(x))\,(n+1)\int\limits_{\frac{k}{n+1}}^{\frac{k+1}{n+1}} \Bigg(f_\mu(x) + f_\mu'(x) \left( t - x \right)\\[0.5em]
&\hspace{5mm}+\left( t - x \right)^2 \left[\frac{1}{2} f_\mu''(x) + \psi\left( t - x \right)\right]\Bigg)\,dt\\[0.5em]
&=\,\ln_\mu(x) f_\mu(x) \sum_{k=0}^n p_{n,k}(a_{n+1}(x))\\
&\hspace{5mm}+ \ln_\mu(x)f_\mu'(x) \sum_{k=0}^n p_{n,k}(a_{n+1}(x))(n+1)\int\limits_{\frac{k}{n+1}}^{\frac{k+1}{n+1}}(t-x)\,dt \\[0.5em]
&\hspace{5mm} +\ln_\mu(x) \frac{1}{2} f_\mu''(x) \sum_{k=0}^n p_{n,k}(a_{n+1}(x))(n+1)\int\limits_{\frac{k}{n+1}}^{\frac{k+1}{n+1}}(t-x)^2 dt\\[0.5em]
&\hspace{5mm}+\ln_\mu(x)\sum_{k=0}^n p_{n,k}(a_{n+1}(x))(n+1)\int\limits_{\frac{k}{n+1}}^{\frac{k+1}{n+1}}(t-x)^2\, \psi(t-x) dt.
\end{align*}
Now, taking into account that $\sum\limits_{k=0}^n p_{n,k}(a_{n+1}(x))=1$, the first term is equal to $\ln_\mu(x)f_\mu(x)=f(x)$. Denoting the last term, i.e., the remainder, by \(\ln_\mu(x)R_n(x)\) for brevity, the operator becomes 
\begin{align*}
\mathscr{L}^K_n(f,x)=&f(x)+\ln_\mu(x)\Bigg[ f_\mu'(x) \sum_{k=0}^n p_{n,k}(a_{n+1}(x))(n+1)\int\limits_{\frac{k}{n+1}}^{\frac{k+1}{n+1}}(t-x)dt\end{align*}\begin{align*}
\hspace{25mm} +\frac{1}{2} f_\mu''(x) \sum_{k=0}^n p_{n,k}(a_{n+1}(x))(n+1)\int\limits_{\frac{k}{n+1}}^{\frac{k+1}{n+1}}(t-x)^2dt + R_n(x)\Bigg],
\end{align*}
so that 
\begin{align*}
\mathscr{L}^K_n(f,x)-f(x)&=\ln_\mu(x)\Bigg[ f_\mu'(x) \sum_{k=0}^n p_{n,k}(a_{n+1}(x))(n+1)\int\limits_{\frac{k}{n+1}}^{\frac{k+1}{n+1}}(t-x)dt\\[0.5em]
& \quad +\frac{1}{2} f_\mu''(x) \sum_{k=0}^n p_{n,k}(a_{n+1}(x))(n+1)\int\limits_{\frac{k}{n+1}}^{\frac{k+1}{n+1}}(t-x)^2dt+ R_n(x)\Bigg] \\[0.5em] &=:\ln_\mu(x) \Big[I_1+I_2+R_n(x)\Big].
\end{align*}
We now analyze the terms appearing in the above expression. First, we consider  
\begin{align*}
I_1&=f_\mu'(x) \sum_{k=0}^n p_{n,k}(a_{n+1}(x))(n+1)\left(\frac{k+1/2}{(n+1)^2}-\frac{x}{n+1}\right)\\[0.5em] 
&=f_\mu'(x) \sum_{k=0}^n p_{n,k}(a_{n+1}(x))\left(\frac{k}{n+1}+\frac{1}{2(n+1)}-x\right)\\[0.5em]
&=f_\mu'(x)\Bigg[\frac{1}{2(n+1)}-x+\sum_{k=0}^n \frac{k}{(n+1)}p_{n,k}(a_{n+1}(x))\Bigg],
\end{align*} where
\begin{align*}
    \sum_{k=0}^n \frac{k}{(n+1)}\,p_{n,k}(a_{n+1}(x))=\frac{n}{n+1}\sum_{k=0}^n \frac{k}{n}\,p_{n,k}(a_{n+1}(x))=\frac{n}{n+1}\,a_{n+1}(x),
\end{align*}
using the known results about King-type operators (see \cite{king} or equation (12) of \cite{Angeloni2025}), so that,
\begin{align*}
I_1&=f_\mu'(x)\Bigg[\frac{1+2\,n\,a_{n+1}(x)}{2(n+1)}-x\Bigg]= \frac{f_\mu'(x)}{(n+1)}\Bigg({\frac{1}{2}+n(a_{n+1}(x)-x)-x}\Bigg).
\end{align*}
As a consequence,
\begin{align*}
    \lim_{n\to +\infty}nI_1={f_\mu'(x)}\Bigg({\frac{1}{2}+\frac{x-x^2}{2(1+\mu)}-x}\Bigg),
\end{align*}
by Lemma \ref{lemma_n(an+1-x)}. Next, computing the integral in $I_2$, we obtain
\begin{align*}
I_2&=\frac{1}{2} f_\mu''(x)\sum_{k=0}^n p_{n,k}(a_{n+1}(x))\Bigg[\frac13 \cdot \frac{3k^2+3k+1}{(n+1)^2}-\frac{2kx+x}{n+1}+x^2\Bigg]\\[0.5em]
&=\frac{1}{2} f_\mu''(x)\Bigg[\frac{1}{3(n+1)^2}+\sum_{k=0}^n \left(\frac{k}{n+1}-x\right)^2p_{n,k}(a_{n+1}(x))\\[0.5em]
&\hspace{5mm}{+\sum_{k=0}^n \frac{k}{(n+1)^2}\,\,p_{n,k}(a_{n+1}(x))-\frac{x}{n+1}}\Bigg],
\end{align*} 
In particular, using again the known results about King-type operators (see \cite{king} or equations (11), (12) and (13) of \cite{Angeloni2025}), we have
\begin{align*}
\sum_{k=0}^n &\left( \frac{k}{n+1} - x \right)^2 p_{n,k}(a_{n+1}(x)) 
\\[1mm]&= \sum_{k=0}^n\frac{k^2}{(n+1)^2}\,\, p_{n,k}(a_{n+1}(x)) -2 x \sum_{k=0}^n \frac{k}{n+1}\,\,p_{n,k}(a_{n+1}(x)) +x^2\\[1mm]
&=\frac{n\,a_{n+1}(x)}{(n+1)^2} + \frac{n(n-1)}{(n+1)^2} a_{n+1}^2(x) - 2x\frac{n}{n+1}a_{n+1}(x) + x^2\\[0.5em]
&=\frac{n^2a_{n+1}^2(x)}{(n+1)^2}- 2x\frac{n}{n+1}a_{n+1}(x) + x^2+\frac{n\,a_{n+1}(x)(1-a_{n+1}(x))}{(n+1)^2}\\[0.5em]
&=\left(\frac{n\,a_{n+1}(x)}{n+1}-x\right)^2+ \frac{n\,a_{n+1}(x)(1-a_{n+1}(x))}{(n+1)^2}.
\end{align*}
So, we obtain that
\begin{align*}
I_2&=\frac{f_\mu''(x)}{2} \Bigg[\frac{1}{3(n+1)^2}-\frac{x}{n+1}+\left(\frac{n\,a_{n+1}(x)}{n+1}-x\right)^2\\[0.5em]&\hspace{5mm}+ \frac{n\,a_{n+1}(x)(2-a_{n+1}(x))}{(n+1)^2}\Bigg]=:\frac{f_\mu''(x)}{2}Q_n(x).
\end{align*}
Now, we study the asymptotic behaviour of $n I_2=\dfrac{f_\mu''(x)}{2}\,n\,Q_n(x)$, as $n \to +\infty$. Analyzing the terms separately and using Lemma \ref{lemma_n(an+1-x)}, we have that
\begin{align*}
&\hspace{-55mm}\lim_{n\to +\infty}\frac{n}{3(n+1)^2}=0,\\[0.5em]
&\hspace{-55mm}\lim_{n\to +\infty}-x\Big(\frac{n}{n+1}\Big)=-x,
\end{align*}\begin{align*}
\lim_{n\to +\infty}n\left(\frac{n a_{n+1}(x)}{n+1}-x\right)^2
\hspace{0.5mm}&=\lim_{n\to +\infty}n\,\left(\frac{n(a_{n+1}(x)-x)}{n+1}-\frac{x}{n+1}\right)^2\\[0.5em]
&=\lim_{n\to +\infty}\frac{n}{(n+1)^2}\Big(n(a_{n+1}(x)-x)-x\Big)^2\\[0.5em]
&=\lim_{n\to +\infty}\frac{n}{(n+1)^2}\lim_{n\to +\infty}\Big(n(a_{n+1}(x)-x)-x\Big)^2\\[0.5em]
&=\left(\frac{x-x^2}{2(1+\mu)}-x \right)^2\lim_{n\to +\infty}\frac{n}{(n+1)^2}=0,
\end{align*} and
\begin{align*}
\lim_{n\to +\infty}\frac{n^2}{(n+1)^2}\,a_{n+1}(x)(2-a_{n+1}(x))=x(2-x),
\end{align*}
since $a_{n+1}(x)\to x$ as $n\to+\infty$. Then, we obtain that
\begin{align}\label{nQn}
    \lim_{n\to{+}\infty} n\,{Q_n}(x)=-x + x(2-x)=x-x^2,
\end{align}
and so,
\begin{align*}
\lim_{n\to{+}\infty} n\,I_2 = f_\mu''(x)\frac{x-x^2}{2}.
\end{align*}
Now, we consider the last term \(R_n{(x)}\), given by
\begin{align*}
    R_n(x) = \sum_{k=0}^n p_{n,k}(a_{n+1}(x))(n+1)\int\limits_{\frac{k}{n+1}}^\frac{k+1}{n+1} (t - x)^2 \psi(t - x) dt.
\end{align*}
Since the function \(\psi(y)\) is bounded and tends to zero as $y\to 0^+$, fixed \(\varepsilon > 0\), there exists \(\widetilde{\delta} > 0\) such that, for every \(y\) with \(0<y \le \widetilde{\delta}\), then \(|\psi(y)| < \varepsilon\). As a consequence, we can write
\begin{align*}
    |R_n(x)| &\le \sum_{k=0}^n p_{n,k}(a_{n+1}(x))(n+1)\int\limits_{\frac{k}{n+1}}^\frac{k+1}{n+1} (t - x)^2 \,|\psi(t - x)| 
    dt\\[2mm]
    &\le \Bigg\{ \sum_{\left|\frac{k}{n+1} - a_{n+1}(x)\right| \le \frac{\widetilde{\delta}}{2}} + \sum_{\left|\frac{k}{n+1} - a_{n+1}(x)\right| > \frac{\widetilde{\delta}}{2}} \Bigg\} \,p_{n,k}(a_{n+1}(x)) \\[2mm]&\qquad \cdot\,(n+1)\int\limits_{\frac{k}{n+1}}^\frac{k+1}{n+1} (t - x)^2 \,|\psi(t - x)| dt \, =: R_1 + R_2.
\end{align*}
To estimate \(R_1\) we observe that, for \(n\) sufficiently large, 
\begin{align*}
|t-x|&\le\left| t-\frac{k}{n+1}\right|+\left| \frac{k}{n+1} - a_{n+1}(x) \right| + \left| a_{n+1}(x) - x \right| \\[0.5em]
&\le\frac{1}{n+1}+\left| \frac{k}{n+1} - a_{n+1}(x) \right| + \left| a_{n+1}(x) - x \right| \leq \widetilde{\delta},
\end{align*}
and so \(\left| \psi\left( t - x \right) \right| < \varepsilon\). Therefore, using the result \eqref{nQn}, there exists a constant \(D > 0\) such that, for sufficiently large $n\in\mathbb{N},$
\begin{align*}
    nR_1 \leq \varepsilon \cdot n \sum_{k=0}^n p_{n,k}(a_{n+1}(x))(n+1)\int\limits_{\frac{k}{n+1}}^{\frac{k+1}{n+1}}(t-x)^2 dt =\varepsilon\cdot n\,{Q_n}(x)\le \varepsilon \, D.
\end{align*}
To estimate \(R_2\), we proceed as follows:
\begin{align*}
    R_2 &\le \|\psi\|_\infty \sum_{\left|\frac{k}{n+1} - a_{n+1}(x)\right| > \frac{\widetilde{\delta}}{2}} 
    p_{n,k}(a_{n+1}(x))(n+1)\int\limits_{\frac{k}{n+1}}^{\frac{k+1}{n+1}} (t-x)^2 dt \\[0.5em]
    &\le \|\psi\|_\infty\hspace{-3mm} \sum_{\left|\frac{k}{n+1} - a_{n+1}(x)\right| > \frac{\widetilde{\delta}}{2}} \hspace{-3mm}
    p_{n,k}(a_{n+1}(x))(n+1)\int\limits_{\frac{k}{n+1}}^{\frac{k+1}{n+1}} \left(t-\frac{k}{n+1}+\frac{k}{n+1}-x\right)^2 dt\\[0.5em]
    & \le \|\psi\|_\infty \sum_{\left|\frac{k}{n+1} - a_{n+1}(x)\right| > \frac{\widetilde{\delta}}{2}} 
    p_{n,k}(a_{n+1}(x))(n+1)\int\limits_{\frac{k}{n+1}}^{\frac{k+1}{n+1}} \Bigg[2\left(t-\frac{k}{n+1}\right)^2\\[0.5em]
    &\hspace{5mm}+2\left(\frac{k}{n+1}-x\right)^2\Bigg] dt\\[-0.5em]
    & \le \|\psi\|_\infty \sum_{\left|\frac{k}{n+1} - a_{n+1}(x)\right| > \frac{\widetilde{\delta}}{2}} 
    p_{n,k}(a_{n+1}(x))(n+1)\int\limits_{\frac{k}{n+1}}^{\frac{k+1}{n+1}} \Bigg[ 2\left(t-\frac{k}{n+1}\right)^2\\[0.5em]
    &\hspace{5mm}+4\left(\frac{k}{n+1} - a_{n+1}(x)\right)^2+ {4}\big(a_{n+1}(x)-x\big)^2 \Bigg] dt\\[0.5em]
    &=\, \|\psi\|_\infty\, \sum_{\left|\frac{k}{n+1} - a_{n+1}(x)\right| > \frac{\widetilde{\delta}}{2}}  
    p_{n,k}(a_{n+1}(x))
    \left[ \frac{2}{{3}(n+1)^2}+4\left(\frac{k}{n+1} - a_{n+1}(x)\right)^2\right.  \\[0.5em]
    &\hspace{5mm}+ {4}\big(a_{n+1}(x)-x\big)^2\Bigg]\end{align*}\begin{align*}
    &= 2\|\psi\|_\infty \Bigg\{ \frac{1}{3(n+1)^2}+{2}(a_{n+1}(x)-x)^2+{2}\sum_{\left|\frac{k}{n+1} - a_{n+1}(x)\right| > \frac{\widetilde{\delta}}{2}}p_{n,k}(a_{n+1}(x))\\[0.5em]
    &\hspace{5mm}\cdot\left(\frac{k}{n+1} - a_{n+1}(x)\right)^2 \Bigg\},
\end{align*} where in the above computations we used twice the inequality $(a+b)^2\le2a^2+2b^2$, for every $a,b \in \mathbb{R}$. As a consequence, we have
\begin{align*}
nR_2 \le &\; 2\|\psi\|_\infty \Bigg\{ \frac{n}{{3}(n+1)^2}+{2}n(a_{n+1}(x)-x)^2+2n
\\[0.2em]&\cdot\sum_{\left|\frac{k}{n+1} - a_{n+1}(x)\right| > \frac{\widetilde{\delta}}{2}} \,p_{n,k}(a_{n+1}(x))\left(\frac{k}{n+1} - a_{n+1}(x)\right)^2 \Bigg\},
\end{align*}
First, we note that the first term in the right-hand side of the above inequality tends to 0, as $n \to +\infty$. 
The estimate of the second term is a consequence of Lemma \ref{lemma_an+1} and Lemma \ref{lemma_n(an+1-x)}, from which $n(a_{n+1}(x)-x)^2$ tends to $0$, as $n \to +\infty.$ 
Finally, to estimate the last term in the right-hand side of the above inequality, we use known estimates of 
$p_{n,k}$ (see \cite{BerType, Angeloni2025}). Indeed, for each \( \widetilde{\delta} > 0 \), there exists a constant $C=C({\widetilde{\delta}})$, such that
\[
\sum_{ \left| \frac{k}{n+1} - a_{n+1}(x) \right| > \frac{\widetilde{\delta}}{2} } p_{n,k}(a_{n+1}(x)) \leq C\, {(n+1)}^{-2}.
\] 
Since \( \left( \frac{k}{n+1} - a_{n+1}(x) \right)^2 \le {2+2\,||a_{n+1}||_\infty^2} \) for all \(0\le\, k \le n \), it follows that
\begin{align*}
&4\| \psi \|_\infty \, n \sum_{ \left| \frac{k}{n+1} - a_{n+1}(x) \right| > \frac{\widetilde{\delta}}{2}} \left( \frac{k}{n+1} - a_{n+1}(x) \right)^2 p_{n,k}(a_{n+1}(x)) \\[0.4em]
\hspace{5mm}& \le 4\| \psi \|_\infty\, n \,\,(2+2\,||a_{n+1}||_\infty^2)\sum_{ \left| \frac{k}{n+1} - a_n(x) \right| > \frac{\widetilde{\delta}}{2}} p_{n,k}(a_{n+1}(x)) \\[0.5em]
&\le 8(1+||a_{n+1}||_\infty^2)\|\psi\|_\infty C\frac{n}{(n+1)^2},
\end{align*}
which again tends to $0$, as $n \to +\infty$. Therefore, for \(n\) sufficiently large, we have that 
$nR_2\le \varepsilon\,A$, for some $A>0$, and consequently
\begin{align*}
    n|R_n(x)|\le \varepsilon(D+A).
\end{align*}
By the results obtained for each term, we finally derive
\begin{align*}
\lim_{n \to +\infty} n\left[ \mathscr{L}_n^K(f,x) - f(x) \right] 
&= \ln_\mu(x)\left[f_\mu'(x)\left(\frac{1}{2}+\frac{x-x^2}{2(1+\mu)}-x\right)\right.\\[0.5em]&\left.\hspace{5mm}+f_\mu''(x)\left(\frac{x-x^2}{2}\right)\right].
\end{align*} This completes the proof.
\\\end{proof}

\section{Saturation results and inverse theorems}
Using the Voronovskaja-type formula given in Theorem \ref{Vor}, we find the second-order differential operator for the Kantorovich operator defined by
\begin{align} \label{D(f)}
{\cal D}(f)(x) := \ln_\mu(x)\left[f_\mu'(x)\left(\frac{1}{2}+\frac{x-x^2}{2(1+\mu)}-x\right)+f_\mu''(x)\left(\frac{x-x^2}{2}\right)\right], 
\end{align}
where $x \in [0,1],$ \(f:[0,1]\to \mathbb{R}\) and \(f_\mu = f/\ln_\mu\) as usual. It is easy to check that it is possible to rewrite the above differential operator in the following standard form
\begin{align} \label{FormaD}
{\cal D}(f) = \frac{1}{w_2}\left( \frac{1}{w_1}\left( \frac{f(x)}{w_0} \right)'\right)',
\end{align}
where
$$
w_0(x)\, :=\, \ln_\mu(x), \quad \quad \quad w_1(x)\, :=\, \frac{1}{x-x^2}\,e^{-x/(\mu+1)}, 
$$
$$
w_2(x)\, :=\, \frac{2}{\ln_\mu(x)}\, e^{x/(\mu+1)}, \quad \quad 0<x<1.
$$
Indeed, if we set:
\[
A(x) := \frac12 + \frac{x-x^2}{2(1+\mu)} - x,\quad B(x):=\frac{x-x^2}{2},
\]
then 
\begin{align*}
{\cal D}(f)(x) &=\ln_\mu(x)\left[A(x)f'_\mu(x) + B(x)f''_\mu(x)\right]\\
&=\frac{1}{w_1 w_2}f_\mu''(x)-\frac{w_1'}{w_2 w_1^2}f_\mu'(x),
\end{align*}with $f_\mu(x)=\frac{f}{\ln_\mu}=\frac{f}{w_0}$.
From this, 
\begin{align*}
  w_1 w_2=\frac{2}{\ln_\mu(x)(x-x^2)}  \quad &\Rightarrow \quad w_2=\frac{2}{\ln_\mu(x)(x-x^2)}\frac{1}{w_1}\\[0.5em]
  -\frac{w_1'}{w_2 w_1^2}=\ln_\mu(x)A(x)\quad &\Rightarrow \quad \frac{w_1'}{w_1}=-A(x)\left(\frac{2}{x-x^2}\right)\\&\hspace{15mm}=-\frac{1}{x-x^2}- \frac{1}{1+\mu} + \frac{2}{1-x}.
\end{align*}
Integrating this last expression, we obtain
\begin{align*}
\ln(w_1(x) )&= -\ln|x|+\ln|1-x|-\frac{x}{1+\mu} 
-2\ln|1-x|\\[0.5em]
&= -\Big[\ln|x|+\ln|1-x|\Big]-\frac{x}{1+\mu}.
\end{align*}
Therefore,
\begin{align*}
    w_1(x) = \frac{1}{x-x^2}\,e^{-x/(1+\mu)}, \quad w_2(x) 
= \frac{2}{\ln_\mu(x)}e^{x/(1+\mu)}.
\end{align*}
These functions \(w_i\), with \(i=0,1,2\), are strictly positive and sufficiently regular in $(0,1)$, and they form an extended complete Tchebychev system and a fundamental set of solutions of the homogeneous equation $D(f)=0$ (see \cite{KAST1966}).

Following the so-called “parabola technique” originally proposed by Bajsanski and Bojanic in \cite{BABO1964}, Garrancho and C\'ardenas-Morales in \cite{GACA2010} obtained a useful lemma related to differential operators expressed in the form~(\ref{FormaD}). We report the lemma below.

\begin{lemma} \label{Lemma-GCM}
Let $J$ be any fixed open interval of $[0,1]$. Let $g,h \in C(J)$ and $t_0 < t_1 < t_2$ be points in $J$ such that 
$g(t_1)=g(t_2)=0$ and $g(t_0)>0$. Then, there exist $\alpha<0$, a function $z$ which is a classical solution of the differential equation 
${\cal D}(f)=0$ on $J$, and a point $x\in(t_1,t_2)$ such that, for all $t\in[t_1,t_2]$,
\[
\alpha\, h(t) + z(t) \,\ge\, g(t), 
\]
and at the point $x$,
\[
\alpha\, h(x) + z(x) = g(x).
\]
\end{lemma}

We now provide a saturation theorem for the operators $\mathscr{L}_n^K$.

\begin{theorem}
Let $f \in C([0,1])$ be fixed. Then
\[
\left| \mathscr{L}_n^K(f,x) - f(x) \right| = o(n^{-1}), 
\qquad 0<x<1,
\]
if and only if $f$ is a classical solution of the differential equation
\begin{align} \label{Eq=0}
A(x)\, f_\mu'(x) + B(x)\, f_\mu''(x) = 0,
\end{align}
where
\[
A(x) := \frac{1}{2} + \frac{x-x^2}{2(1+\mu)} - x, 
\qquad 
B(x) := \frac{x-x^2}{2}.
\]
Equivalently, the equation \eqref{Eq=0} can be expressed in the explicit form
\[
\big(1+\mu - x(1+2\mu) - x^2\big)\, f_\mu'(x)
+ (1+\mu)\,(x-x^2)\, f_\mu''(x) = 0, 
\] with $0<x<1.$
\end{theorem}
The proof proceeds analogously to that one given in \cite{Angeloni2025} for the operators $\mathscr{L}_n$, 
using the Lemma~\ref{Lemma-GCM} and the definition~\eqref{D(f)}.
\vspace{2mm}\\
In conclusion, an inverse approximation theorem for the operators ${\mathscr L}_n^K$ holds, derived directly from Proposition 2 of \cite{GACA2010}.

\begin{theorem}
Let $f \in C([0,1])$ be fixed. Then
\[
n\left| {\mathscr L}_n^K (f, x)-f(x) \right| \ \le \ M + o(1), 
\quad 0<x<1, \quad \text{as } n \to +\infty,
\]
if and only if, for almost every $t \in (0,1)$,
\[
\left|\big(1+\mu - x(1+2\mu) - x^2\big)\, f_\mu'(x)
+ (1+\mu)\,(x-x^2)\, f_\mu''(x)\right| \ \le \ M.
\]
\end{theorem}

\section{Quantitative estimates}
In this section, we derive quantitative estimates for the order of approximation of the operators $\mathscr{L}_n^K$. 

For this purpose, we recall the classical notion of the modulus of continuity of a function $f \in C([0,1])$, defined by
\begin{align*}
    \omega(f,\delta):=\sup_{\substack{x,t \in [0,1]\\[1mm]|t-x|\le\delta}}|f(t)-f(x)|,
\end{align*}
for $\delta >0$. We can prove the following.
\begin{theorem}
For every $n\in \mathbb{N},\, n>1,$ and $f\in C([0,1])$, there holds
\begin{align*}
    \|\mathscr{L}_n^Kf-f\|_\infty\le &\,\omega\left(f_\mu,\frac{1}{\sqrt{n+1}}\right)\ln(2+\mu)\left\{1+\frac{1}{2\sqrt{n+1}}+\sqrt{2}\right\}\\[2mm]
    & + \omega\left(f_\mu,\gamma_n\right)\ln(2+\mu),
\end{align*} where
\begin{align}\label{gamma_n}
    \gamma_n:=\max\limits_{x\in[0,1]} |a_{n+1}(x)-x|.
\end{align}
\end{theorem}
\begin{proof}
We note that
\begin{align*}
    \mathscr{L}_n^K f(x)-f(x)&=\mathscr{L}_n^K f(x)-f(x)\sum_{k=0}^n p_{n,k}(a_{n+1}(x))(n+1)\int\limits_{\frac{k}{n+1}}^\frac{k+1}{n+1}dt\\[2mm]
    &=\mathscr{L}_n^K f(x)-\ln_\mu(x)\sum_{k=0}^n p_{n,k}(a_{n+1}(x))(n+1)\int\limits_{\frac{k}{n+1}}^\frac{k+1}{n+1}f_\mu(x)dt
    \end{align*}\begin{align*}
    \hspace{30mm}=\ln_\mu(x)\sum_{k=0}^n p_{n,k}(a_{n+1}(x))(n+1)\int\limits_{\frac{k}{n+1}}^\frac{k+1}{n+1}\Big[f_\mu(t)-f_\mu(x)\Big]dt,
\end{align*} for every fixed $x\in [0,1].$ Using the properties of the modulus of continuity (see, e.g., \cite{ConApp}), we obtain
\begin{align*}
    |\mathscr{L}_n^K f(x)-f(x)|&\le \ln_\mu(x)\sum_{k=0}^n p_{n,k}(a_{n+1}(x))(n+1)\int\limits_{\frac{k}{n+1}}^\frac{k+1}{n+1}|f_\mu(t)-f_\mu(x)|\,dt\\[0.5em]
    &\le \ln_\mu(x)\sum_{k=0}^n p_{n,k}(a_{n+1}(x))(n+1)\int\limits_{\frac{k}{n+1}}^\frac{k+1}{n+1} \omega(f_\mu,|t-x|)\,dt,
\end{align*} where
\begin{align*}
|t-x|\le |t-a_{n+1}(x)|+|a_{n+1}(x)-x|\le |t-a_{n+1}(x)|+\gamma_n,
\end{align*} with $\gamma_n:=\max\limits_{x\in [0,1]} |a_{n+1}(x)-x|$, $n\in\mathbb{N}$. We note that $\gamma_n$ tends to $0$, as $n \to +\infty,$ by Lemma \ref{lemma_an+1}. 
Therefore,
\begin{align*}
    |\mathscr{L}_n^K f(x)-f(x)|&\le \ln_\mu(x)\sum_{k=0}^n p_{n,k}(a_{n+1}(x))(n+1)\int\limits_{\frac{k}{n+1}}^\frac{k+1}{n+1} \omega(f_\mu,|t-a_{n+1}(x)|)\,dt\\&\hspace{5mm}+\ln_\mu(x)\sum_{k=0}^n p_{n,k}(a_{n+1}(x))(n+1)\int\limits_{\frac{k}{n+1}}^\frac{k+1}{n+1} \omega(f_\mu,\gamma_n)\,dt\\&=:I_1+I_2.
\end{align*} 
For the estimation of $I_2$, we immediately have
\begin{align*}
    I_2 \le \omega(f_\mu,\gamma_n)\ln_\mu(x)\sum_{k=0}^n p_{n,k}(a_{n+1}(x))\le \omega(f_\mu,\gamma_n)\ln(2+\mu).
\end{align*}
To estimate $I_1$, we compute
\begin{align}\label{StimaInt_t-an+1}
    \int\limits_{\frac{k}{n+1}}^\frac{k+1}{n+1} |t-a_{n+1}(x)|dt&\le \int\limits_{\frac{k}{n+1}}^\frac{k+1}{n+1} \left|t-\frac{k}{n+1}\right|dt+\int\limits_{\frac{k}{n+1}}^\frac{k+1}{n+1} \left|\frac{k}{n+1}-a_{n+1}(x)\right|dt\notag\\[2mm]
    &=\frac{1}{2(n+1)^2}+\frac{1}{n+1}\left|\frac{k}{n+1}-a_{n+1}(x)\right|,
\end{align} and we recall the well-known property:
\begin{align*}
    \omega(f,\lambda \delta)\le (1+\lambda)\,\omega(f,\delta), \qquad \lambda,\delta>0.
\end{align*}
As a consequence, we have
\begin{align*}
I_1 &\le \omega\left(f_\mu,\frac{1}{\sqrt{n+1}}\right)\ln_\mu(x)\sum_{k=0}^n p_{n,k}(a_{n+1}(x))(n+1)\\[0.5em]
&\hspace{5mm}\cdot\int\limits_{\frac{k}{n+1}}^\frac{k+1}{n+1} \Big(1+\sqrt{n+1}\,|t-a_{n+1}(x)|\Big)\,dt\\[2mm]
&\le \omega\left(f_\mu,\frac{1}{\sqrt{n+1}}\right)\ln_\mu(x) \Bigg\{1+ (n+1)\sqrt{n+1}\,\sum_{k=0}^n p_{n,k}(a_{n+1}(x))\\[2mm]
&\hspace{5mm}\cdot \int\limits_{\frac{k}{n+1}}^\frac{k+1}{n+1} |t-a_{n+1}(x)|\,dt\Bigg\}\\[2mm]
&\le \omega\left(f_\mu,\frac{1}{\sqrt{n+1}}\right)\ln_\mu(x) \Bigg\{1+ (n+1)\sqrt{n+1}\,\sum_{k=0}^n p_{n,k}(a_{n+1}(x))\\[2mm]
&\hspace{5mm}\cdot \int\limits_{\frac{k}{n+1}}^\frac{k+1}{n+1} \left(\left|t-\frac{k}{n+1}\right|+\left|\frac{k}{n+1}-a_{n+1}(x)\right|\right)dt\Bigg\}\\[2mm]
&=\omega\left(f_\mu,\frac{1}{\sqrt{n+1}}\right)\ln_\mu(x) \Bigg\{1+ (n+1)\sqrt{n+1}\,\sum_{k=0}^n p_{n,k}(a_{n+1}(x))\\[2mm]
&\hspace{5mm}\cdot \left(\frac{1}{2(n+1)^2}+\frac{1}{n+1}\left|\frac{k}{n+1}-a_{n+1}(x)\right|\right)\Bigg\}\\[2mm]
&\le \omega\left(f_\mu,\frac{1}{\sqrt{n+1}}\right)\ln_\mu(x) \Bigg\{1+\frac{\sqrt{n+1}}{2(n+1)}+\sqrt{n+1}\,\sum_{k=0}^n p_{n,k}(a_{n+1}(x))\\[2mm]
&\hspace{5mm}\cdot\left|\frac{k}{n+1}-a_{n+1}(x)\right|\Bigg\}\\[2mm]
&\le \omega\left(f_\mu,\frac{1}{\sqrt{n+1}}\right)\ln(2+\mu)\left\{1+\frac{1}{2\sqrt{n+1}}+\sqrt{2}\right\},
\end{align*} using the estimate \eqref{DisLemma5.1} in Lemma~5.1 of \cite{Angeloni2024}, for every $x\in [0,1]$ and $n>1$. Finally, by combining the estimates of $I_1$ and $I_2$, we obtain the desired result.\\
\end{proof}
Now, we provide a quantitative estimate in $L^p$ spaces. 
To this end, we make use of the Peetre $K$–functional (see, e.g., \cite{ConApp}), defined for $f\in L^p([0,1])$ by
\begin{align*}
    \mathcal{K}(f,t)_p := \inf_{g\in C^1([0,1])} 
    \Big\{\, \|f-g\|_p + t\,\|g\|_{C^1} \Big\}, 
    \qquad t>0,\; 1\le p<+\infty,
\end{align*}
where $C^1([0,1])$ denotes the space of functions $g:[0,1]\to \mathbb{R}$ with a continuous first derivative, endowed with the norm
\[
\|g\|_{C^1} := \|g\|_\infty + \|g'\|_\infty.
\]
Clearly, $C^1([0,1])$ is continuously embedded in $L^p([0,1])$. 

Moreover, it is well known that there exists a constant $C>0$ such that
\begin{align*}
    \omega(f,t)_p \le C\, \mathcal{K}(f,t)_p, 
    \qquad t>0,\; 1\le p<+\infty,
\end{align*}
where $\omega(f,t)_p$ denotes the first-order modulus of smoothness in $L^p([0,1])$. For completeness, we recall that the $r$-th modulus of smoothness is given by
\begin{align*}
    \omega_r(f,t)_p
    := \sup_{0<h\le t} \|\Delta_h^r (f,\cdot)\|_p(A_{r h}),
    \qquad  t\ge 0,
\end{align*}
where $r=1,2,\dots,h\ge0$, $A_{r h} := [0,1-rh]$, and $\Delta_h^r$ is the $r$-th order forward difference, defined by
\[
\Delta_h^r(f,x) := \sum_{k=0}^r (-1)^{r-k} \binom{r}{k} f(x+k h), 
\]
for all $x \in A_{rh}$ (see, e.g., \cite{ConApp}). We have the following.
\begin{theorem}\label{ThPeetre1}
    Let $f\in L^p([0,1])$, $1\le p < +\infty$. Then, for every sufficiently large $n\in\mathbb{N}$, there holds
    \begin{align*}
    \|\mathscr{L}_n^K f - f\|_p
    &\le\frac{\ln(2+\mu)}{\ln(1+\mu)} (K_\mu+1)\,
    \mathcal{K}\!\left(f,\Lambda_n\right)_p,
\end{align*}
where
\begin{align*}
\Lambda_n &:= \frac{\ln(2+\mu)}{\ln(1+\mu)}\Bigg(1+\frac{1}{(1+\mu)\ln(1+\mu)}\Bigg)(K_\mu+1)^{-1}\,T_n,
\end{align*}with\begin{align} \label{Tn}
T_n&:=\frac{1}{2(n+1)}+\frac{\sqrt{2}}{\sqrt{n+1}}+\gamma_n,
\end{align} 
$K_\mu$ is the constant arising from Theorem \ref{th_Kmu} and $\gamma_n$ is defined in \eqref{gamma_n}.
\end{theorem}
\begin{proof}
For every fixed $f\in L^p([0,1])$, $1\le p < +\infty$, and $g \in C^1([0,1])$, we have that
\begin{align*}
    \|\mathscr{L}_n^Kf-f\|_{p,\mu} &\le \|\mathscr{L}_n^Kf-\mathscr{L}_n^Kg\|_{p,\mu}+\|\mathscr{L}_n^Kg-g\|_{p,\mu}+\|g-f\|_{p,\mu}\\[0.5em]
    & \le (K_\mu+1)\|g-f\|_{p,\mu}+\|\mathscr{L}_n^Kg-g\|_{p,\mu}\\[0.5em]
    &\le \frac{(K_\mu+1)}{\ln(1+\mu)}\,\|g-f\|_p+\frac{1}{\ln(1+\mu)}\|\mathscr{L}_n^Kg-g\|_{p},
\end{align*} using Theorem \ref{th_Kmu} and the estimate \eqref{stima_normap}. Therefore, 
\begin{align}\label{DisNormap}
     \|\mathscr{L}_n^Kf-f\|_{p}\le \frac{\ln(2+\mu)}{\ln(1+\mu)} \Big[(K_\mu+1)\|g-f\|_p+\|\mathscr{L}_n^Kg-g\|_{p}\Big].
\end{align}
In particular, we consider
\begin{align} \label{LnKg-g}
    |\mathscr{L}_n^Kg(x)-g(x)|&=|\mathscr{L}_n^Kg(x)-g(x)\mathscr{L}_n^Ke_0(x)+g(x)\mathscr{L}_n^Ke_0(x)-g(x)|\notag\\[0.5em]
    &\le |\mathscr{L}_n^Kg(x)-g(x)\mathscr{L}_n^Ke_0(x)|+|g(x)|\cdot|\mathscr{L}_n^Ke_0(x)-e_0(x)|,\notag\\[0.3em]
\end{align}where
\begin{align} \label{LnKg-gLnKe0}
    \mathscr{L}_n^Kg(x)&-g(x)\mathscr{L}_n^Ke_0(x)\notag\\[-0.2em]&=\ln_\mu(x)\sum_{k=0}^np_{n,k}(a_{n+1}(x))(n+1)\int\limits_{\frac{k}{n+1}}^\frac{k+1}{n+1}\frac{g(t)-g(x)}{\ln_\mu(t)} dt\notag\\[0.2em]
    &=\ln_\mu(x)\sum_{k=0}^np_{n,k}(a_{n+1}(x))(n+1)\int\limits_{\frac{k}{n+1}}^\frac{k+1}{n+1}g'(\xi_{t,x})\,\frac{(t-x)}{\ln_\mu(t)}dt,
\end{align} where $\xi_{t,x}$ indicates a suitable point between $t$ and $x$, and so
\begin{align*}
    |\mathscr{L}_n^Kg(x)-g(x)\mathscr{L}_n^Ke_0(x)|&\le \|g'\|_\infty\ln_\mu(x)\sum_{k=0}^np_{n,k}(a_{n+1}(x))(n+1)\\[0.5em]&\hspace{5mm}\cdot\int\limits_{\frac{k}{n+1}}^\frac{k}{n+1}\frac{|t-x|}{\ln_\mu(t)}dt\\[0.5em]&\le\|g'\|_\infty\frac{\ln(2+\mu)}{\ln(1+\mu)}\sum_{k=0}^np_{n,k}(a_{n+1}(x))\Bigg\{(n+1)\end{align*}\begin{align*}
    \hspace{45mm}\cdot\int\limits_{\frac{k}{n+1}}^\frac{k}{n+1}|t-a_{n+1}(x)|dt+|a_{n+1}(x)-x|\Bigg\}.
\end{align*} Using the inequalities \eqref{StimaInt_t-an+1} and \eqref{DisLemma5.1}, we obtain that
\begin{align} \label{|LnKg-gLnKe0|}
|\mathscr{L}_n^Kg(x)-g(x)\mathscr{L}_n^Ke_0(x)|&\le\|g'\|_\infty\frac{\ln(2+\mu)}{\ln(1+\mu)}\sum_{k=0}^np_{n,k}(a_{n+1}(x))\left\{\frac{1}{2(n+1)}\right.
\notag\\[0.5em]&\hspace{5mm}\left.+\left|\frac{k}{n+1}-a_{n+1}(x)\right|+\gamma_n \right\}\notag\\[0.5em]
&\le\|g'\|_\infty\frac{\ln(2+\mu)}{\ln(1+\mu)}\left\{\frac{1}{2(n+1)}+\frac{\sqrt{2}}{\sqrt{n+1}}+\gamma_n\right\},\notag\\[1mm]
\end{align} where $\gamma_n:=\max\limits_{x\in[0,1]}|a_{n+1}(x)-x|$, for every $n>1$. On the other hand, we get
\begin{align*}
    |g(x)|\cdot|\mathscr{L}_n^Ke_0(x)-e_0(x)|\le&\,|g(x)|\cdot \ln_\mu(x)\sum_{k=0}^n p_{n,k}(a_{n+1}(x))(n+1)\\[0.5em]&\cdot\int\limits_{\frac{k}{n+1}}^\frac{k+1}{n+1}\left|\frac{1}{\ln_\mu(t)}-\frac{1}{\ln_\mu(x)}\right|dt.
\end{align*} We note that the function $\frac{1}{\ln_\mu(\cdot)}$ is of class $C^1$ on $[0,1]$, hence it is Lipschitz continuous. In particular, there exists a constant $L=L(\mu)>0$ such that 
\begin{align*}
    \left|\frac{1}{\ln_\mu(t)}-\frac{1}{\ln_\mu(x)}\right|
    \;\le\; L\,|t-x|, \qquad t,x\in[0,1],
\end{align*}
where
$L=(1+\mu)^{-1}\ln^{-2}(1+\mu).$ Therefore, by the fundamental theorem of calculus we have
\begin{align*}
    |g(x)|\cdot|\mathscr{L}_n^Ke_0(x)-e_0(x)|
    \le&\, \,\Big(||g||_\infty+||g'||_\infty\Big)\frac{\ln_\mu(x)}{(1+\mu)\ln^2(1+\mu)}\\[0.5em]
    &\cdot\sum_{k=0}^n p_{n,k}(a_{n+1}(x))(n+1)\int\limits_{\tfrac{k}{n+1}}^{\tfrac{k+1}{n+1}} |t-x|\, dt.
\end{align*}
Now, we estimate the integral. For every $t\in\big[\tfrac{k}{n+1},\tfrac{k+1}{n+1}\big]$, writing as before
\begin{align*}
    |t-x|\;\le\;\Big|t-\tfrac{k}{n+1}\Big|
    +\Big|\tfrac{k}{n+1}-a_{n+1}(x)\Big|
    +|a_{n+1}(x)-x|,
\end{align*}
\begin{align*}
    (n+1)\int\limits_{\tfrac{k}{n+1}}^{\tfrac{k+1}{n+1}} |t-x|\,dt
    &\le (n+1)\int\limits_{\tfrac{k}{n+1}}^{\tfrac{k+1}{n+1}}\Big|t-\tfrac{k}{n+1}\Big|dt
    +\Big|\tfrac{k}{n+1}-a_{n+1}(x)\Big|\\[0.5em]
    &\hspace{5mm}+|a_{n+1}(x)-x|\\[0.5em]
    &\le \frac{1}{2(n+1)}+\Big|\tfrac{k}{n+1}-a_{n+1}(x)\Big|+\gamma_n,
\end{align*}
where $\gamma_n$ is defined in \eqref{gamma_n}. Now, using \eqref{DisLemma5.1} by Lemma 5.1 of \cite{Angeloni2024}, we obtain 
\begin{align}\label{g*LnKe0-e0}
    |g(x)|\cdot|\mathscr{L}_n^Ke_0(x)-e_0(x)|\le \frac{\,\ln(2+\mu)}{(1+\mu)\ln^2(1+\mu)}T_n\,\Big( ||g||_\infty+||g'||_\infty\Big),
\end{align} where \begin{align*} 
    T_n:=\frac{1}{2(n+1)}+\frac{\sqrt{2}}{\sqrt{n+1}}+\gamma_n,
\end{align*} and consequently,
\begin{align*}
    ||\mathscr{L}^Kg-g||_p&\le \frac{\ln(2+\mu)}{\ln(1+\mu)}\Bigg(1+\frac{1}{(1+\mu)\ln(1+\mu)}\Bigg)T_n\,||g'||_\infty\\[0.5em]
    &\hspace{5mm}+\frac{\ln(2+\mu)}{(1+\mu)\ln^2(1+\mu)}\,T_n ||g||_\infty\\[0.5em]
    &\le\frac{\ln(2+\mu)}{\ln(1+\mu)}\Bigg(1+\frac{1}{(1+\mu)\ln(1+\mu)}\Bigg)T_n\,\Big(||g||_\infty+||g'||_\infty\Big).
\end{align*} 
By the previous estimates, we deduce that
\begin{align*}
    \|\mathscr{L}_n^K f - f\|_p
    \le &\,\frac{\ln(2+\mu)}{\ln(1+\mu)}\Bigg[(K_\mu+1)\|f-g\|_p\\[0.5em]&
      + \frac{\ln(2+\mu)}{\ln(1+\mu)}\Bigg(1+\frac{1}{(1+\mu)\ln(1+\mu)}\Bigg)\,T_n\,\|g\|_{C^1}\Bigg].
\end{align*}
Passing to the infimum with respect to $g\in C^1([0,1])$ and using the definition of the Peetre $K$-functional, we get
\begin{align*}
    \|\mathscr{L}_n^K f - f\|_p
    &\le \frac{\ln(2+\mu)}{\ln(1+\mu)} (K_\mu+1)\,
    \mathcal{K}\!\left(f,\Lambda_n\right)_p,
\end{align*}
where
\begin{align*}
    \Lambda_n := \frac{\ln(2+\mu)}{\ln(1+\mu)}\Bigg(1+\frac{1}{(1+\mu)\ln(1+\mu)}\Bigg)(K_\mu+1)^{-1}\,T_n,
\end{align*}for every $n \in \mathbb{N}$ sufficiently large.
\\
\end{proof}

In Theorem~\ref{ThPeetre1} we derived a quantitative estimate in terms of the $K$-functional $\mathcal{K}(f,\delta)_p$, considering functions in the subspace $C^1([0,1])$ endowed with the $\|\cdot\|_{C^1}$ norm. 

Now we introduce a slightly different form of the $K$-functional, defined for $f \in L^p([0,1])$ by
\[
\widetilde{\mathcal{K}}(f,t)_p :=
\inf_{g\in W^{1,p}([0,1])}
\Big\{ \|f-g\|_p + t\,\|g\|_{W^{1,p}} \Big\},
\qquad t>0,\; 1\le p < +\infty.
\]
Here $W^{1,p}([0,1])$ denotes the Sobolev space of all the absolutely continuous functions 
$g:[0,1]\to\mathbb{R}$ with first derivative $g'\in L^p([0,1])$, equipped with the norm
\[
\|g\|_{W^{1,p}} := \|g\|_p + \|g'\|_p.
\]
Thus, in the definition of $\widetilde{\mathcal{K}}(f,t)_p$ we use the Sobolev space $W^{1,p}([0,1])$ instead of $C^1([0,1])$. The advantage is that this definition is slightly more general, allowing us to derive quantitative approximation estimates for the operators $\mathscr{L}_n^K$ in terms of the $K$-functional $\widetilde{\mathcal{K}}(f,t)$. To this purpose, it will be useful to recall the notion of Hardy--Littlewood maximal function, defined by
\[
M(f;x) := \sup_{\substack{0<t\le 1\\[1mm]t\neq x}}\,\, \frac{1}{t-x}\int\limits_{x}^t |f(u)|\,du,
\]
which satisfies (see, e.g., \cite{stein1970singular})
\begin{align}\label{Cp}
    \|M(f;\cdot)\|_p \le C_p \|f\|_p,\quad  1<p<+\infty,
\end{align}
for $f\in L^p([0,1])$, and for a suitable constant $C_p>0$ depending only on $p$.
\begin{theorem}\label{ThPeetre2}
For every sufficiently large $n\in \mathbb{N}$ and $f \in L^p([0,1])$, $1<p<+\infty$, there holds
\begin{align*}
\|\mathscr{L}^K_nf-f\|_p  \le \frac{\ln(2+\mu)}{\ln(1+\mu)}(K_\mu+1)\widetilde{K}(f,\Gamma_n)_p,
\end{align*}
where
\begin{align*}
\Gamma_n:=2^{\frac{p-1}{p}}\,C_{p,\mu}\,\frac{\ln(2+\mu)}{\ln(1+\mu)}\,(K_\mu+1)^{-1}\,T_n,
\end{align*} $K_\mu$ is the constant arising from Theorem \ref{th_Kmu}, $T_n$ is defined in \eqref{Tn} and $$C_{p,\mu}:=\max\left\{\frac{1}{(1+\mu)\ln(1+\mu)};C_p\right\}$$
with $C_p$ defined in \eqref{Cp}.
\end{theorem}
\begin{proof}
Let $g\in AC[0,1]$, with $g' \in L^p([0,1])$, $1<p<+\infty$, be fixed. Using the same computations made in \eqref{LnKg-g} and \eqref{LnKg-gLnKe0}, we similarly obtain
\begin{align*}
    |\mathscr{L}_n^Kg(x)-g(x)|&\le |g(x)|\cdot|\mathscr{L}_n^Ke_0(x)-e_0(x)|+\ln_\mu(x)\\[0.5em]
    &\hspace{5mm}\cdot\sum_{k=0}^np_{n,k}(a_{n+1}(x))(n+1)\int\limits_{\frac{k}{n+1}}^\frac{k}{n+1}\frac{1}{\ln_\mu(t)}\left|\int\limits_x^tg'(u)du\right|dt\\[0.5em]
    &\le |g(x)|\cdot|\mathscr{L}_n^Ke_0(x)-e_0(x)|+\ln(2+\mu)\,|M(g';x)|\\[0.5em]
    &\hspace{5mm}\cdot\sum_{k=0}^np_{n,k}(a_{n+1}(x))(n+1)\int\limits_{\frac{k}{n+1}}^\frac{k}{n+1} \frac{|t-x|}{\ln_\mu(t)}\,dt,
\end{align*} for sufficiently large $n\in \mathbb{N}$. Therefore, using the inequalities \eqref{|LnKg-gLnKe0|} and \eqref{g*LnKe0-e0}, we deduce that
\begin{align*}
    |\mathscr{L}^K_ng(x)-g(x)|&\le |g(x)| \frac{\ln(2+\mu)}{(1+\mu)\ln^2(1+\mu)}\left\{\frac{1}{2(n+1)}+\frac{\sqrt{2}}{\sqrt{n+1}}+\gamma_n\right\}\\[0.5em] &\qquad+|M(g';x)|\frac{\ln(2+\mu)}{\ln(1+\mu)}\left\{\frac{1}{2(n+1)}+\frac{\sqrt{2}}{\sqrt{n+1}}+\gamma_n\right\}\\[0.5em]
    &= \frac{\ln(2+\mu)}{\ln(1+\mu)}\Bigg(\frac{1}{(1+\mu)\ln(1+\mu)}|g(x)|+|M(g';x)|\Bigg)T_n\,,
\end{align*}where \(T_n\) is defined in \eqref{Tn}. Now, for every fixed $f \in L^p([0,1])$, using \eqref{DisNormap} and then the above inequalities together with \eqref{Cp}, we obtain
\begin{align*}
    \|\mathscr{L}_n^Kf-f\|_p&\le\, \frac{\ln(2+\mu)}{\ln(1+\mu)} \Bigg[(K_\mu+1)||f-g||_p+2^{\frac{p-1}{p}}\frac{\ln(2+\mu)}{\ln(1+\mu)}\\[0.5em]
    &\qquad\cdot\left(\frac{1}{(1+\mu)\ln(1+\mu)}\|g\|_p+C_p\|g'\|_p\right)T_n\Bigg]\\[0.5em]
    &\le \frac{\ln(2+\mu)}{\ln(1+\mu)} \Bigg[(K_\mu+1)||f-g||_p+2^{\frac{p-1}{p}}\frac{\ln(2+\mu)}{\ln(1+\mu)}\\[0.5em]&\qquad \cdot C_{p,\mu}\Big(||g||_p+||g'||_p\Big)T_n\Bigg],
\end{align*}
where 
$$C_{p,\mu}:=\max\left\{\frac{1}{(1+\mu)\ln(1+\mu)};C_p\right\}.$$
Finally, passing to the infimum with respect to $g\in W^{1,p}([0,1])$ and using the definition of $K$-functional $\widetilde{\mathcal{K}}(f,t)_p$, we immediately obtain the thesis, that is,
\begin{align*}
    \|\mathscr{L}_n^Kf-f\|_p\le\frac{\ln(2+\mu)}{\ln(1+\mu)}(K_\mu+1)\,\widetilde{\mathcal{K}}(f,\Gamma_n)_p,
\end{align*} where 
\begin{align*}
    \Gamma_n:=2^{\frac{p-1}{p}}\,C_{p,\mu}\,\frac{\ln(2+\mu)}{\ln(1+\mu)}\,(K_\mu+1)^{-1}\,{T_n},
\end{align*} 
for $n\in \mathbb{N}$ sufficiently large.\\
\end{proof}

\begin{remark}
\normalfont 
The quantity $T_n$ appearing in theorems \ref{ThPeetre1} and \ref{ThPeetre2} is defined in \eqref{Tn}, with
\[
\gamma_n:=\max_{x\in[0,1]}\big(a_{n+1}(x)-x\big),
\]
where $a_{n+1}$ is given in \eqref{an+1}. 
We point out that $T_n = \mathcal{O}\Big(\frac{1}{\sqrt{n}}\Big)$, as $n \to +\infty$. 
In order to study its asymptotic behaviour, and in particular that one of $\gamma_n$, we let
\[
\varepsilon_{n+1} := \frac{1}{(n+1)(1+\mu)}.
\]
so that $\varepsilon_{n+1}\to 0^{+}$ as \( n \to +\infty \). Accordingly, we rewrite
\[
a_{n+1}(x)
=\frac{\ln(1+x\varepsilon_{n+1})}{\ln(1+\varepsilon_{n+1})}.
\]
Applying the Taylor expansion of the logarithmic function and of the geometric series, we obtain, for $n$ sufficiently large,
\[
a_{n+1}(x)-x
= \varepsilon_{n+1}\,\frac{x-x^{2}}{2} + o(\varepsilon_{n+1}).
\]
Since $0\le x-x^{2}\le 1/4$, 
\[
\gamma_n=\max_{x\in[0,1]} (a_{n+1}(x)-x)
=\mathcal{O}(\varepsilon_{n+1})
=\mathcal{O}\!\left(\frac{1}{n}\right).
\]
This is also consistent with Lemma~\ref{lemma_an+1}, which ensures the uniform convergence \(a_{n+1}(x)\to x\) as \(n\to+\infty\).  
As a consequence,
\[
T_n=\frac{1}{2(n+1)}+\frac{\sqrt{2}}{\sqrt{n+1}}
+\mathcal{O}\!\left(\frac{1}{n}\right)
=\mathcal{O}\!\left(\frac{1}{\sqrt{n}}\right),
\qquad n\to+\infty.
\]
\end{remark}

\section*{Acknowledgements}
{\small The authors are members of the Gruppo Nazionale per l'Analisi Matematica, la Probabilit\`a e le loro Applicazioni (GNAMPA) of the Istituto Nazionale di Alta Matematica (INdAM), of the network RITA (Research ITalian network on Approximation), and of the UMI (Unione Matematica Italiana) group T.A.A. (Teoria dell'Approssimazione e Applicazioni). 
}

\section*{Funding}
{\small The authors L. Angeloni and D. Costarelli have been partially supported within the (1) "National Innovation Ecosystem grant ECS00000041 - VITALITY", funded by the European Union - Next-GenerationEU under the Italian Ministry of University and Research (MUR) and (2) 2025 GNAMPA-INdAM Project "MultiPolExp: Polinomi di tipo esponenziale in assetto multidimensionale e multivoco" (CUP E5324001950001).}

\section*{Conflict of interest/Competing interests}
{\small The author declares that he has no conflict of interest and competing interest.}

\section*{Availability of data and material and Code availability}
{\small Not applicable.}

\bibliographystyle{plain}

\end{document}